\RequirePackage{fix-cm}

\documentclass[smallextended,envcountsame,envcountreset,numbook]{svjour3} 

\smartqed

\usepackage{amsmath}
\usepackage{amssymb}
\usepackage{hyperref}

\newcommand\restr[2]{{
  \left.\kern-\nulldelimiterspace 
  #1 
  \vphantom{\big|} 
  \right|_{#2}
  }}

\mathchardef\mhyphen="2D

\journalname{Journal of Theoretical Probability}

\begin{document}

\clearpage\thispagestyle{empty}

\begin{center}

\textsc{\Large Uniquely determined uniform probability on the natural numbers}\\[1.5cm]

\textsc{\LARGE Timber Kerkvliet and Ronald Meester\footnote{T. Kerkvliet, R. Meester.
	\\ Department of Mathematics, Faculty of Sciences, VU University,
	De Boelelaan 1081a, 1081 HV Amsterdam, The Netherlands.
	E-mail:t.kerkvliet@vu.nl, r.w.j.meester@vu.nl.
}}\\[1.5cm]
\end{center}

\begin{abstract} 
In this paper, we address the problem of constructing a uniform probability measure on $\mathbb{N}$. Of course, this is not possible within the bounds of the Kolmogorov axioms and we have to violate at least one axiom. We define a probability measure as a finitely additive measure assigning probability $1$ to the whole space, on a domain which is closed under complements and finite disjoint unions. We introduce and motivate a notion of uniformity which we call weak thinnability, which is strictly stronger than extension of natural density. We construct a weakly thinnable probability measure and we show that on its domain, which contains sets without natural density, probability is uniquely determined by weak thinnability. In this sense, we can assign uniform probabilities in a canonical way. We generalize this result to uniform probability measures on other metric spaces, including $\mathbb{R}^n$.

\end{abstract}

\section{Introduction and main results}
\label{sec:introduction}

Within the bounds of the Kolmogorov axioms \cite{kolmogorov}, a probability measure on $\mathbb{N} = \{1,2,3,...\}$ cannot assign the same probability to every singleton and therefore, a uniform probability measure on $\mathbb{N}$ does not exist. Despite this, we have some intuition about what a uniform probability measure on $\mathbb{N} $ should look like. According to this intuition, for example, we would assign probability $1/2$ to the subset of all odd numbers. If we want to capture this intuition in a mathematical framework, we have to violate at least one of the axioms of Kolmogorov.

One suggestion by De Finetti \cite{definetti} is to relax countable additivity of the measure to finite additivity. To see why this suggestion is reasonable, we must first understand why it is possible, within the axioms of Kolmogorov, to set up uniform (Lebesgue) measure on $[0,1]$. The type of additivity we demand plays a crucial role here. In the standard theory one always demands \emph{countable} additivity. If every singleton has the same probability, in an infinite space, every singleton must have probability zero. With countable additivity this means that every countable set must have probability zero. This is no problem if we are working on the uncountable $[0,1]$, since we still have freedom to assign different probabilities to different uncountable subsets of $[0,1]$. The interval $[0,1/2]$, for example, has Lebesgue measure $1/2$, while it is equipotent with $[0,1]$, which has Lebesgue measure $1$. This works because the cardinality of the set over which we sum is smaller than the 
cardinality of the space itself.

On $\mathbb{N}$ the problem of countable additivity is immediately clear: since every subset of $\mathbb{N}$ is countable, every subset should have probability zero, which is impossible because the probability of $\mathbb{N}$ itself should be $1$. In analogy with Lebesgue measure, we want finite subsets to have probability zero and we want to be able to assign different probabilities to countable subsets. To do this, we should change the type of additivity to finite additivity. In short: since the cardinality of the space changes from uncountable to countable, the additivity should change from countable to finite.

Schirokauer and Kadane \cite{schirokauerkadane} study three different collections of finitely additive probability measures on $\mathbb{N}$ which may qualify as uniform: the set $L$ of measures that extend natural density, the set $S$ of shift-invariant measures and the set $R$ of measures that measure residue classes uniform. They show that $N \subset S \subset R$ where the inclusions are strict. If a set $A \subseteq \mathbb{N}$ is without natural density, i.e. 
\begin{equation}
\frac{|A \cap \{1,2,\ldots,n\}|}{n}
\end{equation}
does not converge as $n \rightarrow \infty$, different measures in $L$ assign different probabilities to $A$. So even the smallest collection discussed by Schirokauer and Kadane does not lead to a uniquely determined uniform probability for sets which do not have a natural density. This observation brings us to the main goal of this paper.

\medskip\noindent
\textbf{Main goal:} find a natural notion of uniformity, stronger than extension of natural density, such that the collection of all probability measures that are uniform under this notion, assign the same probability to a large collection of sets. In particular, this collection of sets should be larger than the collection of sets having a natural density.

\medskip
In this paper, we introduce and study a notion of uniformity which is stronger than the extension of natural density. A uniform probability measure on $[0,1]$ or on a finite space is characterised by the property that if we condition on any suitable subset, the resulting conditional probability measure is again uniform on that subset. It is this property that we will generalise, and the generalised notion will be called {\em weak thinnability}. (The actual definition of weak thinnability is given later, and will also involve two technical conditions.) 

We allow probability measures to be defined on collections of sets that are closed under complements and finite disjoint unions. This is because we think there is no principal reason to insist that all sets are measured, just like not all subsets of $\mathbb{R}$ are Lebesgue measurable. We should, however, be cautious when allowing  domains that are not necessarily algebras, for the following reason. De Finetti \cite{definetti} uses a Dutch Book argument to conclude that, under the Bayesian interpretation of probability, a probability measure has to be coherent. He shows that if the domain of the probability measure is an algebra, the finite additivity of the probability measure implies coherence. On domains only closed under complements and finite disjoint unions, however, this implication no longer holds. Therefore, someone sharing de Fenitti's view of probability, would like to add coherence as additional constraint. For completeness, we study both the case with and the case without coherence as additional constraint on the probability measure.

\begin{definition}
\label{def:probabilitypair}
Let $X$ be a space and write $\mathcal{P}(X)$ for the power set of $X$. An \emph{$f$-system} on $X$ is a nonempty collection $\mathcal{F} \subseteq \mathcal{P}(X)$ such that
\begin{description}
\item[1] $A,B \in \mathcal{F}$ with $A \cap B = \emptyset$ implies that $A \cup B \in \mathcal{F}$,
\item[2] $A \in \mathcal{F}$ implies that $A^c \in \mathcal{F}$.
\end{description}
A \emph{probability measure} on an $f$-system $\mathcal{F}$ is a map $\mu: \mathcal{F} \rightarrow [0,1]$ such that
\begin{description}
\item[1]  $A,B \in \mathcal{F}$ with $A \cap B = \emptyset$ implies that $\mu(A \cup B) = \mu(A)+\mu(B)$,
\item[2]  $\mu(X)=1$.
\end{description}
A \emph{coherent} probability measure is a probability measure $\mu: \mathcal{F} \rightarrow [0,1]$ such that for all $n \in \mathbb{N}$, $\alpha_1,...,\alpha_n \in \mathbb{R}$, $A_1,...,A_n \in \mathcal{F}$
\begin{equation}
\label{eq:coherence}
\sup_{x \in X} \sum_{i=1}^n \alpha_i ( I_{A_i}(x) - \mu(A_i)) \geq 0.
\end{equation}
A \emph{probability pair} on $X$ is a pair $(\mathcal{F},\mu)$ such that $\mathcal{F}$ is an $f$-system on $X$ and $\mu$ is a probability measure on $\mathcal{F}$.
\end{definition}

\begin{remark}
Schurz and Leitgeb \cite[p.~261]{schurzleitgeb} call an $f$-system a pre-Dynkin system, since in case of closure under \emph{countable} unions of mutually disjoint sets, such a collection is called a Dynkin system.
\end{remark}

\begin{remark}
Expression \ref{eq:coherence} has the following interpretation. If $\alpha_i \geq 0$, we buy a bet on $A_i$ that pays out $\alpha_i$ for $\alpha\mu(A_i)$. If $\alpha_i<0$, we sell a bet on $A_i$ that pays out $|\alpha_i|$ for $|\alpha_i|\mu(A_i)$. Then (\ref{eq:coherence}) expresses there is no guaranteed amount of net loss.
\end{remark}

We aim at uniquely determining the probability of as many sets as possible. In particular, we are interested in probability pairs with an $f$-system consisting only of sets with a uniquely determined probability. So we are not only interested in probability pairs satisfying our stronger notion of uniformity, but in the {\em canonical} ones, where ``canonical" is to be understood in the following way.

\begin{definition}
\label{def:uniquevalues}
Let $P$ be some collection of probability pairs. A pair $(\mathcal{F},\mu) \in P$ is canonical with respect to $P$ if for every $A \in \mathcal{F}$ and every pair $(\mathcal{F}',\mu') \in P$ with $A \in \mathcal{F}'$  we have $\mu(A)=\mu'(A)$. 
\end{definition}

Before we give a more detailed outline of our paper, we need the following definition. Set
\begin{equation}
\mathcal{M} := \left\{ \bigcup_{i=1}^\infty [a_{2i-1},a_{2i}) \;\;:\;\; 0 \leq a_1 \leq a_2 \leq a_3 \leq \cdots \right\}.
\end{equation}
Note that $\mathcal{M}$ is an algebra on $[0,\infty)$. It turns out that by working on $[0,\infty)$ instead of $\mathbb{N}$, where we restrict ourselves to sub-$f$-systems of $\mathcal{M}$, we can formulate and prove our claims much more elegantly. Here, we view the elements of $\mathcal{P}(\mathbb{N})$ embedded in $\mathcal{M}$ by the injection
\begin{equation}
A \mapsto \bigcup_{n \in A} [n-1,n).
\end{equation}
We should emphasize, however, that conceptually there is no difference between $[0,\infty)$ and $\mathbb{N}$ and that the work we do in Sections \ref{sec:uniformity} and \ref{sec:pair} can be done in the same way for $\mathbb{N}$. After working on $\mathcal{M}$, we explicitly translate our result to $\mathbb{N}$ and other metric spaces in Section \ref{sec:general}.

For $A \in \mathcal{M}$ we define $\rho_A: [0,\infty) \rightarrow [0,1]$ by $\rho_A(0):=0$ and
\begin{equation}
\rho_A(x) := \frac{1}{x} \int_0^x 1_A(y) \mathrm{d}y
\end{equation}
for $x>0$. Also set
\begin{equation}
\mathcal{C} := \left\{ A \in \mathcal{M} \;:\; \rho_A(x) \;\mathrm{converges} \right\},
\end{equation}
which are the elements of $\mathcal{M}$ that have natural density and let $\lambda: \mathcal{C} \rightarrow [0,1]$ be given by
\begin{equation}
\lambda(A) := \lim_{x \rightarrow \infty} \rho_A(x). 
\end{equation}
We write $L^*$ for the collection of probability pairs $(\mathcal{F},\mu)$ on $[0,\infty)$ such that $\mathcal{C} \subseteq \mathcal{F} \subseteq \mathcal{M}$ and $\mu(A)=\lambda(A)$ for $A \in \mathcal{C}$. Our earlier observation about the indeterminacy of probability under $L$ gets the following formulation in terms of $L^*$: a pair $(\mathcal{F},\mu) \in L^*$ is canonical with respect to $L^*$ if and only if $\mathcal{F} = \mathcal{C}$. We write $WT$ for the collection of probability pairs that are a \emph{weakly thinnable pair} (WTP), that is, a probability pair that satisfies the condition of weak thinnability. The collection $WT$ is a proper subset of $L^*$ and contains pairs $(\mathcal{F},\mu)$ canonical with respect to $WT$ such that $\mathcal{F} \setminus \mathcal{C} \not= \emptyset$. In other words, with restricting $L^*$ to $WT$ we are able to assign a uniquely determined probability to some sets without natural density. Finally, we write $WTC \subseteq WT \subset L^*$ for the elements $(\mathcal{F},\mu) \in WT$ such that $\mu$ is coherent.

The structure of this paper is as follows. In Section \ref{sec:uniformity}, we discuss weak thinnability and motivate why this is a natural notion of uniformity. In Section \ref{sec:pair}, we introduce the probability pair $(\mathcal{A}^\mathrm{uni},\alpha)$ where
\begin{equation}
\label{eq:introAuni}
\mathcal{A}^\mathrm{uni} = \left\{ A \in \mathcal{M} \;:\; \exists L \;\; \lim_{D \rightarrow \infty} \sup_{x\in(1,\infty)} \left| \frac{1}{\log(D)} \int_x^{xD} \frac{1_A(y)}{y} \mathrm{d}y - L \right| = 0 \right\}
\end{equation}
and
\begin{equation}
\label{eq:introalpha}
\alpha(A) = \lim_{D \rightarrow \infty}  \frac{1}{\log(D)} \int_1^D \frac{1_A(y)}{y} \mathrm{d}y.
\end{equation}
\begin{remark}
The expression in (\ref{eq:introalpha}) is sometimes called the logarithmic density of $A$ \cite[p.~272]{tenenbaum}.
\end{remark}
We end Section \ref{sec:pair} with the following theorem, which is the main result of our paper. 
\begin{theorem}[Main theorem]
\label{thm:main}
The following holds:
\begin{itemize}
\item[1] The pair $(\mathcal{A}^\mathrm{uni},\alpha)$ is a WTP, is extendable to a WTP $(\mathcal{F},\mu)$ with $\mathcal{F}=\mathcal{M}$ and $\alpha$ is coherent.
\item[2] The pair $(\mathcal{A}^\mathrm{uni},\alpha)$ is canonical with respect to both $WT$ and $WTC$.
\item[3] If a pair $(\mathcal{F},\mu)$ is canonical with respect to $WT$ or $WTC$, then $\mathcal{F} \subseteq \mathcal{A}^\mathrm{uni}$.
\end{itemize}
\end{theorem}

In Section \ref{sec:general}, we derive from $(\mathcal{A}^\mathrm{uni},\alpha)$ analogous probability pairs on certain metric spaces including Euclidean space. The proofs of the results in Sections \ref{sec:uniformity}-\ref{sec:general} are given in Section \ref{sec:proofs}.

We write $\mathbb{N}_0:= \{0,1,2,...\}$. For real-valued sequences $x,y$ or real-valued functions $x,y$ on $[0,\infty)$ we write $x \sim y$ or $x_i \sim y_i$ if $\lim_{i \rightarrow \infty} (x_i-y_i)=0$. Since we work only on $[0,\infty)$ in Sections \ref{sec:uniformity} and \ref{sec:pair}, every time we speak of an $f$-system, probability pair or probability measure it is understood that this is on $[0,\infty)$.

\section{Weak thinnability}
\label{sec:uniformity}

Let $m$ be the Lebesgue measure on $\mathbb{R}$. For Lebesgue measurable $Y \subseteq \mathbb{R}$ with $0<m(Y)<\infty$ the uniform probability measure on $Y$ is given by
\begin{equation}
\mu_{Y}(X) := \frac{m(X)}{m(Y)}
\end{equation}
for all Lebesgue measurable $X \subseteq Y$. Let $A \subseteq B \subseteq C$ be all Lebesgue measurable with $m(B)>0$ and $m(C)<\infty$. Observe that 
\begin{equation}
\label{eq:finiterelation}
\mu_C(A) = \mu_C(B) \mu_B(A).
\end{equation}

We want to generalize this property to a property of probability pairs on $[0,\infty)$. For $A \in \mathcal{M}$ define $S_A: [0,\infty) \rightarrow [0,\infty)$ by
\begin{equation}
S_A(x) := m(A \cap [0,x)).
\end{equation}
Write
\begin{equation}
\mathcal{M}^* := \left\{ A \in \mathcal{M} \;:\; m(A) = \infty \right\}.
\end{equation}
Consider for $A \in \mathcal{M}^*$  the map $f_A: A \rightarrow [0,\infty)$ given by $f_A(x):=S_A(x)$. The map $f_A$ gives a one-to-one correspondence between $A$ and $[0,\infty)$. If $A \in \mathcal{M}^*$ and $B \in \mathcal{M}$, we want to introduce notation for the set
\begin{equation}
\{ f_A^{-1}(b) \;:\; b \in B \},
\end{equation}
that gives the subset of $A$ that corresponds to $B$ under $f_A$. Inspired by van Douwen \cite{vandouwen}, we introduce the following operation.

\begin{definition}
\label{def:thinnability}
For $A,B \in \mathcal{M}$, define
\begin{equation}
A \circ B := \{ x \in [0,\infty) \;:\; x \in A \;\wedge\;S_A(x) \in B \}.
\end{equation}
\end{definition}
Note that if $A,B \in \mathcal{M}$, then $A \circ B \in \mathcal{M}$ and that for $A \in \mathcal{M}^*$ we have
\begin{equation}
A \circ B = \{ f_A^{-1}(b) \;:\; b \in B \}.
\end{equation}
We can view this operation as thinning $A$ by $B$ because we create a subset of $A$, where $B$ is ``deciding'' which parts of $A$ are removed. We also can view the operation $A \circ B$ as thinning out $B$ over $A$, since we ``spread out'' the set $B$ over $A$. Taking for example
\begin{equation}
A = \bigcup_{i=0}^\infty [2i,2i+1) = [0,1) \cup [2,3) \cup [4,5) \cup [6,7) \cup ...
\end{equation}
and
\begin{equation}
B = \bigcup_{i=1}^\infty [i^2-1,i^2) =[0,1) \cup [3,4) \cup [8,9) \cup [15,16) \cup ...
\end{equation}
we get
\begin{equation}
A\circ B = [0,1) \cup [6,7) \cup [16,17) \cup [30,31) \cup [48,49) \cup [70,71) \cup ...
\end{equation}
and
\begin{equation}
B\circ A = [0,1) \cup [8,9) \cup [24,25) \cup [48,49) \cup [80,81] \cup [120,121) \cup ...
\end{equation}

Let $(\mathcal{F},\mu)$ be a probability pair and let $A \in \mathcal{F} \cap \mathcal{M}^*$. If $B \in \mathcal{M}$, the set $A\circ B$ is the subset of $A$ corresponding to $B$. We can use this to transform $\mu$ into a measure on $A$ as follows. We set $\mathcal{F}_A := \{ A \circ B \;:\; B \in \mathcal{F} \}$ and then define $\mu_A : \mathcal{F}_A \rightarrow [0,1]$ by
\begin{equation}
\label{eq:derivativemeasure}
\mu_A(A \circ B) := \mu(B).
\end{equation}
Given $B \in \mathcal{F}$ such that $A \circ B \in \mathcal{F}$, the condition that
\begin{equation}
\label{eq:infiniterelation1}
\mu(A\circ B) = \mu(A) \mu_A(A \circ B)
\end{equation}
is a natural generalization of (\ref{eq:finiterelation}). Using (\ref{eq:derivativemeasure}) this translates into
\begin{equation}
\label{eq:infiniterelation3}
\mu(A\circ B) = \mu(A) \mu(B).
\end{equation}

We now have the restriction that $A \in \mathcal{F} \cap \mathcal{M}^*$. However, if $A \in \mathcal{F} \setminus \mathcal{M}^*$, then any uniform probability measure should assign $0$ to $A$ and since $A\circ B \subseteq A$ (\ref{eq:infiniterelation3}) still holds. In Section \ref{subsec:thinnability}, we show that the condition that (\ref{eq:infiniterelation3}) holds for all $A,B \in \mathcal{F}$ is so strong that only probability pairs with relatively small $f$-systems satisfy it. Since it is our goal to find a notion of uniformity that allows for a canonical pair with a large $f$-system, we choose to use a weakened version of this property which asks that $\mu(C \circ A) = \mu(C)\mu(A)$ for every $C \in \mathcal{C}$ and $A \in \mathcal{F}$.

Weak thinnability also involves two technical conditions. Let $(\mathcal{F},\mu)$ be a probability pair, let $A,B \in \mathcal{F}$ and suppose it is true for every $x \in [0,\infty)$ that 
\begin{equation}
\label{eq:orderingrho}
S_A(x) \geq S_B(x).
\end{equation}
Since this inequality is true \emph{for every} $x$, the set $B$ is ``sparser'' than $A$. Therefore, it is natural to ask that $\mu(A) \geq \mu(B)$. We call this property ``preserving ordering by $S$''.

Since we have $\mathcal{C} \subseteq \mathcal{F}$, it seems natural to also ask $\restr{\mu}{\mathcal{C}}=\lambda$, but it turns out to be sufficient to ask the weaker property that $\mu([c,\infty))=1$ for every $c \in [0,\infty)$. So, to reduce redundancy we require the latter and then prove that $\restr{\mu}{\mathcal{C}}=\lambda$. Putting everything together, we obtain the following definition.

\begin{definition}
\label{def:wtp}
A probability pair $(\mathcal{F},\mu)$ with $\mathcal{F} \subseteq \mathcal{M}$ is a WTP if it satisfies the following conditions:
\begin{description}
\item[P1] For every $C \in \mathcal{C}$ and $A \in \mathcal{F}$ we have $C \circ A \in \mathcal{F}$ and $\mu(C \circ A)=\mu(C)\mu(A)$,
\item[P2] $\mu$ preserves ordering by $S$,
\item[P3] $\mu([c,\infty))=1$ for every $c \in [0,\infty)$.
\end{description}
\end{definition}

That every WTP extends natural density is implied by the following result. 

\begin{proposition}
\label{prop:muandrho}
Let $(\mathcal{F},\mu) \in WT$. Then for $A \in \mathcal{F}$ we have
\begin{equation}
\liminf_{x \rightarrow \infty} \rho_A(x) \leq \mu(A) \leq \limsup_{x \rightarrow \infty} \rho_A(x).
\end{equation}
\end{proposition}

\section{The pair $(\mathcal{A}^\mathrm{uni},\alpha)$}
\label{sec:pair}

For $A \in \mathcal{M}$ set $\sigma_A: (0,\infty)^2 \rightarrow [0,1]$ given by
\begin{equation}
\sigma_A(D,x) := \frac{1}{D} \int_{x}^{x+D} 1_A(y) \mathrm{d}y,
\end{equation}
which is the average of $1_A$ over the interval $[x,x+D]$. Then set for any $A \in \mathcal{M}$
\begin{equation}
U(A) := \limsup_{D \rightarrow \infty} \sup_{x \in (0,\infty)} \sigma_A(D,x)
\end{equation}
and
\begin{equation}
L(A) := \liminf_{D \rightarrow \infty} \inf_{x \in (0,\infty)} \sigma_A(D,x).
\end{equation}
Define
\begin{equation}
\mathcal{W}^\mathrm{uni} := \{ A \in \mathcal{M} \;:\; L(A)=U(A) \}.
\end{equation}
It is easy to check that $(\mathcal{W}^\mathrm{uni},\restr{\lambda}{\mathcal{W}^\mathrm{uni}})$ is a probability pair. For any $A \in \mathcal{M}$, we set
\begin{equation}
\log(A) := \{ \log(a) \;:\; a \in A \cap [1,\infty) \}.
\end{equation}

\begin{definition}
\label{def:Aunialpha}
We define
\begin{equation}
\mathcal{A}^\mathrm{uni} := \{ A \in \mathcal{M} \;:\; \log(A) \in \mathcal{W}^\mathrm{uni} \}
\end{equation}
and $\alpha: \mathcal{A}^\mathrm{uni} \rightarrow [0,1]$ by
\begin{equation}
\alpha(A) := \lambda(\log(A)).
\end{equation}
\end{definition}

Notice that Definition \ref{def:Aunialpha} gives a definition of $(\mathcal{A}^\mathrm{uni},\alpha)$ that is slightly different from (\ref{eq:introAuni}) and (\ref{eq:introalpha}). For a justification of equations \ref{eq:introAuni} and \ref{eq:introalpha}, see the proof of Lemma \ref{lem:rep1}. Our first concern is that $\alpha$ coincides with natural density.
\begin{proposition}
\label{prop:alphaextendslambda}
We have $\mathcal{C} \subseteq \mathcal{A}^\mathrm{uni}$ and for every $A \in \mathcal{C}$
\begin{equation}
\alpha(A) = \lambda(\log(A)) = \lambda(A).
\end{equation}
\end{proposition}
A typical example of a set in $\mathcal{A}^\mathrm{uni}$ that is not in $\mathcal{C}$, is
\begin{equation}
A = \bigcup_{n=0}^\infty [e^{2n},e^{2n+1}).
\end{equation}
It is easy to check that $A \not\in \mathcal{C}$, but
\begin{equation}
\log(A) = \bigcup_{n=0}^\infty [2n,2n+1),
\end{equation}
so $\log(A) \in \mathcal{W}^\mathrm{uni}$ with $\lambda(\log(A))=1/2$. Hence $A \in \mathcal{A}^\mathrm{uni}$ with $\alpha(A)=1/2$.

That $(\mathcal{A}^\mathrm{uni},\alpha)$ is a probability pair follows directly from the fact that $(\mathcal{W}^\mathrm{uni},\restr{\lambda}{\mathcal{W}^\mathrm{uni}})$ is a probability pair. The pair  $(\mathcal{A}^\mathrm{uni},\alpha)$ is also a WTP.
\begin{theorem}
\label{thm:wtp}
We have $(\mathcal{A}^\mathrm{uni},\alpha) \in WTC \subseteq WT$ and we can extend $(\mathcal{A}^\mathrm{uni},\alpha)$ to a $WTP$ with $\mathcal{M}$ as $f$-system.
\end{theorem}

\begin{remark}
We use free ultrafilters in the proof of Theorem \ref{thm:wtp} to show there exists an extension to a $WTP$ ith $\mathcal{M}$ as $f$-system. The existence of free ultrafilters is guaranteed by the Boolean Prime Ideal Theorem, which can not be proven in ZF set theory, but is weaker than the axiom of choice \cite{halpern}. The existence of a atomfree or nonprincipal (i.e. every singleton has measure zero) finite additive measure defined on the power set of $\mathbb{N}$ cannot be established in ZF alone \cite{solovay}. Consequently, a version of the axiom of choice is always necessary to construct a probability measure on $\mathcal{M}$ that assigns measure zero to all bounded intervals.
\end{remark}

We do not only want an element of $WTC$, but a canonical one. This is guaranteed by the following theorem.
\begin{theorem}
\label{thm:unique}
The pair $(\mathcal{A}^\mathrm{uni},\alpha)$ is canonical with respect to both $WT$ and $WTC$.
\end{theorem}

The pair $(\mathcal{A}^\mathrm{uni},\alpha)$ is maximal in the sense that it contains every pair that is canonical with respect to $WT$ or $WTC$.
\begin{theorem}
\label{thm:maximal}
If $(\mathcal{F},\mu) \in WT$ is canonical with respect to $WT$, then $\mathcal{F} \subseteq \mathcal{A}^\mathrm{uni}$. If $(\mathcal{F},\mu) \in WTC$ is canonical with respect to $WTC$, then $\mathcal{F} \subseteq \mathcal{A}^\mathrm{uni}$.
\end{theorem}

\section{Generalization to metric spaces}
\label{sec:general}

In this section we derive probability pairs on a class of metric spaces that are analogous to $(\mathcal{A}^\mathrm{uni},\alpha)$. Of course one could also try to construct such a probability measure by working more directly on these metric spaces, instead of constructing a derivative of $(\mathcal{A}^\mathrm{uni},\alpha)$. Since probability pairs on $[0,\infty)$, motivated from the problem of a uniform probability measure on $\mathbb{N}$, is the priority of this paper, we do not make such an effort here.

Let us first sketch the idea of the generalization. Let $A \in \mathcal{M}$. Whether $A$ is in $\mathcal{A}^\mathrm{uni}$ depends completely on the asymptotic behavior of $\rho_A$ (Lemma \ref{lem:rep1}). If $A \in \mathcal{A}^\mathrm{uni}$, then also $\alpha(A)$ only depends on the asymptotic behavior of $\rho_A$ (Lemma \ref{lem:rep1}). Now suppose that on a space $X$, we can somehow define a density functions $\bar{\rho}_B:[0,\infty)\rightarrow [0,1]$ for (some) subsets $B \subseteq X$ in a canonical way. Then, by replacing $\rho$ by $\bar{\rho}$, we get the analogue of $(\mathcal{A}^\mathrm{uni},\alpha)$ in $X$. The goal of this section is to make this idea precise.

Let $(X,d)$ be a metric space. For $x \in X$ and $r \geq 0$, write 
\begin{equation}
B(x,r) := \{ y \in X \;:\; d(y,x)<r \}.
\end{equation}
Write $\mathcal{B}(X)$ for the Borel $\sigma$-algebra of $X$. We need a ``uniform'' measure on this space to measure density of subsets in open balls. It is clear that the measure of an open ball should at least be independent of where in the space we look, i.e. it should only depend on the radius of the ball. This leads to the following definition.

\begin{definition}
\label{def:uniformmeasure}
We say that a Borel measure $\nu$ on $X$ is uniform if for all $r > 0$ and $x,y \in X$ we have
\begin{equation}
\label{eq:openballsequal}
0< \nu(B(x,r))=\nu(B(y,r))<\infty.
\end{equation}
\end{definition}

On $\mathbb{R}^n$ with Euclidean metric, the standard Borel measure as obtained by assigning to a product of intervals the product of the lengths of those intervals, is a uniform measure. In general, on normed locally compact vector spaces, the invariant measure with respect to vector addition, as given by the Haar measure, is a uniform measure.

A result by Christensen \cite{christensen} tells us that uniform measures that are Radon measures are unique up to multiplicative constants on locally compact metric spaces. This, however, does not cover all cases. The set of irrational numbers, for example, is not locally compact, but the Lebesgue measure restricted to Borel sets of irrational numbers is a uniform measure and unique up to a multiplicative constant. We give a slightly more general version of the result of Christensen.

\begin{proposition}
\label{prop:uniformunique}
If $\nu_1$ and $\nu_2$ are two uniform measures on $X$, then there exists some $c>0$ such that $\nu_1=c \nu_2$.
\end{proposition}

Proposition \ref{prop:uniformunique} gives us uniqueness, but not existence. To see that there are metric spaces without a uniform measure, consider the following example. Let $X$ be the set of vertices in a connected graph that is not regular. Let $d$ be the graph distance on $X$. If we suppose that $\nu$ is a uniform measure on $X$, from (\ref{eq:openballsequal}) with $r<1$ it follows that for some $C>0$ we have $\nu(\{x\})=C$ for every $x \in X$. But then $\nu(B(x,2))=C(1+\deg(x))$ for every $x\in V$, which implies (\ref{eq:openballsequal}) cannot hold for $r=2$ since the graph is not regular. A characterization of metric spaces on which a uniform measure exist, does not seem  to be present in the literature.

We now assume $X$ has a uniform measure $\nu$ and that $\nu(X)=\infty$. In addition to that, we write $h(r):= \nu(B(x,r))$ for $r\geq 0$ and assume that
\begin{equation}
\label{eq:assumptionX}
\forall C>0 \;\;\lim_{r \rightarrow \infty} \frac{h(r+C)}{h(r)}=1,
\end{equation}
which is equivalent with amenability in case $(X,d)$ is a normed locally compact vector space \cite{pier}. For the importance of this assumption, see Remark \ref{rem:asymptotic} below.

Set
\begin{equation}
\begin{aligned}
r^-(u) & := \sup \left\{r \in [0,\infty) \;:\; h(r) \leq u \right\},  \\
r^+(u) & := r^- + 1 \\
\end{aligned}
\end{equation}
for $u \in [0,\infty)$. Note that $h(r^-(u)) \leq u$ and $h(r^+(u)) \geq u$. Write $(X,\mathcal{L}(X),\bar{\nu})$ for the (Lebesgue) completion of $(X,\mathcal{B}(X),\nu)$. Fix some $o \in X$. For $A \in \mathcal{L}(X)$ define the map $\bar{\rho}_A: [0,\infty) \rightarrow [0,\infty)$ given by $\bar{\rho}_A(0):=0$ and
\begin{equation}
\label{eq:rhobar}
\bar{\rho}_A(u) := \frac{\bar{\nu}(B(o,r^-(u)) \cap A)}{h(r^-(u))}
\end{equation}
for $r>0$. The value $\bar{\rho}_A(u)$ is the density of $A$ in the biggest open ball around $o$ of at most measure $u$. Notice that $\bar{\rho}_A$ is independent of the choice of $\nu$ as a result of Proposition \ref{prop:uniformunique}. The function $\bar{\rho}_A$ does depend on the choice of $o$, but in Proposition \ref{prop:asymptotic} we show that the asymptotic behavior of $\bar{\rho}_A$ does not depend on the choice of $o$. We also show in Proposition \ref{prop:asymptotic} that the asymptotic behavior of $\bar{\rho}_A$ is not affected if we replace $r^-(u)$ by $r^+(u)$ in (\ref{eq:rhobar}). 
\begin{proposition}
\label{prop:asymptotic}
Fix $x,y \in  X$ and $A \in \mathcal{L}(X)$. Then
\begin{equation}
\frac{\bar{\nu}(B(x,r^-(u)) \cap A)}{h(r^-(u))} \sim \frac{\bar{\nu}(B(y,r^+(u)) \cap A)}{h(r^+(u))}.
\end{equation}
\end{proposition}

\begin{remark}
\label{rem:asymptotic}
Proposition \ref{prop:asymptotic} is not necessarily true if we do not assume (\ref{eq:assumptionX}), as illustrated by the following example. Suppose $X$ is the set of vertices of a $3$-regular tree graph and $d$ is the graph distance. Let $\nu$ be the counting measure, which is a uniform measure on this metric space. Then clearly (\ref{eq:assumptionX}) is not satisfied. Now pick any $x \in X$ and let $y$ be a neighbor of $x$. Let $A \subseteq \mathcal{P}(X)$ be the connected component containing $y$ in the graph where the edge between $x$ and $y$ is removed. Then
\begin{equation}
\lim_{r \rightarrow \infty} \frac{\bar{\nu}(B(x,r) \cap A)}{h(r)} = 1/3 \;\mathrm{and} \; \lim_{r \rightarrow \infty} \frac{\bar{\nu}(B(y,r) \cap A)}{h(r)} = 2/3.
\end{equation}
\end{remark}

Proposition \ref{prop:asymptotic} justifies the use of $\bar{\rho}$ to determine the density, since its asymptotic behavior is canonical. So, we define for $A \in \mathcal{L}(X)$ the map $\bar{\xi}_A: (1,\infty)^2 \rightarrow [0,1]$ given by
\begin{equation}
\bar{\xi}_A(D,x) :=  \frac{1}{\log(D)} \int_x^{Dx} \frac{\bar{\rho}_A(y)}{y} \mathrm{d}y.
\end{equation}
Then we set
\begin{equation}
\label{eq:AuniX}
\mathcal{A}^\mathrm{uni}(X) := \left\{ A \in \mathcal{L}(X) \;:\; \limsup_{D \rightarrow \infty} \sup_{x>1} \bar{\xi}_A(D,x) = \liminf_{D \rightarrow \infty} \inf_{x>1} \bar{\xi}_A(D,x) \right\}
\end{equation}
and $\alpha^X: \mathcal{A}^\mathrm{uni}(X)  \rightarrow [0,1]$ by
\begin{equation}
\label{eq:alphaX}
\alpha^X(A) := \limsup_{D \rightarrow \infty} \sup_{x \in (1,\infty)} \bar{\xi}_A(D,x) = \liminf_{D \rightarrow \infty} \inf_{x \in (1,\infty)} \bar{\xi}_A(D,x).
\end{equation}
The pair $(\mathcal{A}^\mathrm{uni}(X),\alpha^X)$ gives us the analogue of $(\mathcal{A}^\mathrm{uni},\alpha)$ in $X$. In particular, it gives for $X=\mathbb{N}$ the corresponding uniform probability measure on $\mathbb{N}$ we initially searched for. In case of Euclidean space, we have the following expression for $(\mathcal{A}^\mathrm{uni}(X),\alpha^X)$, which in the special case of $X=\mathbb{R}$ gives us an extension of $\alpha$ ($\mathcal{A}^\mathrm{uni}(\mathbb{R})$ is the maximal sub-$f$-system of $\mathcal{L}(\mathbb{R})$, where $\mathcal{A}^\mathrm{uni} \subseteq \mathcal{A}^\mathrm{uni}(\mathbb{R})$ is the maximal sub-$f$-system of $\mathcal{M}$).

\begin{proposition}
\label{prop:euclideanspace}
Suppose $X=\mathbb{R}^n$ and $d$ is Euclidean distance. Let $\sigma$ be the surface measure on the unit sphere in $\mathbb{R}^n$. Then for $A \in \mathcal{L}(\mathbb{R}^n)$ we can replace $\bar{\xi}_A(D,x)$ in (\ref{eq:AuniX}) and (\ref{eq:alphaX}) by 
\begin{equation}
\frac{1}{\log(D)}\int_x^{Dx} \frac{K_A(y)}{y} \mathrm{d}y,
\end{equation}
where $K_A:[0,\infty) \rightarrow [0,1]$ is given by
\begin{equation}
K_A(r) := \frac{\Gamma(n/2)}{2 \pi^{n/2}} \int_{S^{n-1}} 1_A(ru) \sigma(\mathrm{d}u).
\end{equation}
\end{proposition}

\section{Proofs}
\label{sec:proofs}

First we show that every $f$-system of a WTP is closed under translation and that every probability measure of a WTP is invariant under translation.

\begin{lemma}
\label{lem:trans}
Let $(\mathcal{F},\mu)$ be a WTP. Let $A \in \mathcal{F}$ and $c \in [0,\infty)$. Then
\begin{equation}
A' := \{ c + a \;:\; a \in A \} \in \mathcal{F}
\end{equation}
and $\mu(A)=\mu(A')$.
\end{lemma}

\begin{proof}
Let $(\mathcal{F},\mu)$ be a WTP. Let $A \in \mathcal{F}$ and $c \in [0,\infty)$. Set $B:=[c,\infty)$. We have $B \in \mathcal{C} \subseteq \mathcal{F}$ and by P3 we have $\mu(B)=1$. Therefore, $A' = B \circ A \in \mathcal{F}$ and
\begin{equation}
\mu(A') = \mu(B)\mu(A)=\mu(A)
\end{equation}
by P1.
\qed
\end{proof}

\begin{proof} \textit{of Propositon \ref{prop:muandrho}} \;
Let $(\mathcal{F},\mu)$ be a WTP and $A \in \mathcal{F}$. Set $u := \limsup_{x \rightarrow \infty} \rho_A(x)$. If $u=1$ there is nothing to prove, so assume $u<1$. Let $\epsilon>0$ be given. Let $u' \in [0,1] \cap \mathbb{Q}$ such that $u'>u$ and $u'-u<\epsilon$. The idea is to construct a $Y \in \mathcal{M}$ such that we can easily see that $\mu(Y)=u'$ and $\rho_A(x) \leq \rho_Y(x)$ for all $x$, so that with P2 we get $\mu(A) \leq u'$.

First we observe that there is a $K>0$ such that for all $x\geq K$ we have $\rho_A(x) \leq u'$. We can write $u'$ as $u' = \frac{p}{q}$ for some $p,q \in \mathbb{N}_0$ with $p \leq q$. Now we introduce the set $Y$ given by
\begin{equation*}
Y := [0,K) \cup \bigcup_{i=0}^\infty [iq, iq + p).
\end{equation*}
Note that $Y \in \mathcal{C} \subseteq \mathcal{F}$. Lemma \ref{lem:trans} and the fact that $\mu$ is a probability measure, gives us that $\mu(Y)=u'$.  Further, observe that for each $x \in [0,\infty)$ we have $\rho_A(x) \leq \rho_Y(x)$, so with P2 we get
\begin{equation*}
\mu(A) \leq \mu(Y) = u' < u + \epsilon.
\end{equation*}
Letting $\epsilon \downarrow 0$ we find
\begin{equation*}
\mu(A) \leq u =  \limsup_{x \rightarrow \infty} \rho_A(x).
\end{equation*}
By applying this to $A^c$ we find
\begin{equation*}
\mu(A) = 1 - \mu(A^c) \geq 1 - \limsup_{x \rightarrow \infty} \rho_{A^c}(x) = \liminf_{x \rightarrow \infty} \rho_A(x).
\end{equation*}
\qed
\end{proof}

Before we prove Proposition \ref{prop:alphaextendslambda} and Theorem \ref{thm:wtp}, we present the following alternative representation of $(\mathcal{A}^\mathrm{uni},\alpha)$. We define for $A \in \mathcal{M}$ the map $\xi_A: (1,\infty)^2 \rightarrow [0,1]$ given by
\begin{equation}
\xi_A(D,x) := \frac{1}{\log(D)} \int_{x}^{Dx} \frac{\rho_A(y)}{y} \mathrm{d}y.
\end{equation}

Set
\begin{equation}
\mathcal{S} := \left\{ s \in (1,\infty)^\mathbb{N} \;:\; \lim_{n \rightarrow \infty} s_n = \infty \right\}.
\end{equation}
If $s \in \mathcal{S}$ and $f \in (1,\infty)^\mathbb{N}$, then we can interpret the pair $(s,f)$ as the sequence $(s_1,f_1),(s_2,f_2),\ldots$ in $(1,\infty)^2$. Write 
\begin{equation}
\mathcal{P} := \left\{ (s,f) \;:\; s \in \mathcal{S},\;\; f \in (1,\infty)^\mathbb{N} \right\}
\end{equation}
for the collection of all such sequences. 

For every $(s,f) \in \mathcal{P}$ we set
\begin{equation}
\mathcal{A}^{s,f} := \{ A \in \mathcal{M} \;:\; \lim_{n \rightarrow \infty} \xi_A(s_n,f_n) \; \mathrm{exists} \}
\end{equation}
and
\begin{equation}
\alpha^{s,f}(A) := \lim_{n \rightarrow \infty} \xi_A(s_n,f_n).
\end{equation}

\begin{lemma}[Alternate Representation]
\label{lem:rep1}
We have
\begin{equation}
\label{eq:Auni}
\mathcal{A}^\mathrm{uni} = \bigcap_{(s,f)\in\mathcal{P}} \mathcal{A}^{s,f}
\end{equation}
with for any $(s,f) \in \mathcal{P}$ and $A \in \mathcal{A}^\mathrm{uni}$
\begin{equation}
\label{eq:alpha}
\alpha(A) = \alpha^{s,f}(A).
\end{equation}
\end{lemma}

\begin{proof}
Let $A \in \mathcal{M}$. We start to relate $\sigma_{\log(A)}$ and $\xi_A$. If $D,x \in (1,\infty)$, then
\begin{equation}
\begin{aligned}
\sigma_{\log(A)}(\log(D),\log(x)) & = \frac{1}{\log(D)} \int_{\log(x)}^{\log(Dx)} 1_A(e^y) \mathrm{d}y \\
& = \frac{1}{\log(D)} \int_{x}^{Dx} \frac{1_A(u)}{u} \mathrm{d}u \\
&  = \frac{1}{\log(D)} \int_{x}^{Dx} \frac{ S_A'(u)}{u} \mathrm{d}u \\
&  = \frac{1}{\log(D)} \left( \frac{S_A(u)}{u} \biggr\rvert_{u=x}^{Dx} + \int_{x}^{Dx} \frac{ S_A(u)}{u^2} \mathrm{d}u \right)\\
&  = \frac{\rho_A(Dx)-\rho_A(x)}{\log(D)} + \xi_A(D,x).
\end{aligned}
\end{equation}
This implies that for $(s,f) \in \mathcal{P}$ we have
\begin{equation}
\mathcal{A}^{s,f} = \left\{ A \in \mathcal{M} \;:\; \lim_{n \rightarrow \infty} \sigma_{\log(A)}(\log(s_n),\log(f_n)) \;\mathrm{exists} \right\}
\end{equation}
with for $A \in \mathcal{A}^{s,f}$
\begin{equation}
\alpha^{s,f}(A) = \lim_{n \rightarrow \infty} \sigma_{\log(A)}(\log(s_n),\log(f_n)).
\end{equation}
Since for any $A \in \mathcal{M}$ and $(s,f) \in \mathcal{P}$
\begin{equation}
L(\log(A)) \leq \liminf_{n \rightarrow \infty} \sigma_{\log(A)}(\log(s_n),\log(f_n))
\end{equation}
and
\begin{equation}
\limsup_{n \rightarrow \infty} \sigma_{\log(A)}(\log(s_n),\log(f_n)) \leq U(\log(A)),
\end{equation}
we find that if $\log(A) \in \mathcal{W}^\mathrm{uni}$, then $A \in \mathcal{A}^{s,f}$ with $\alpha^{s,f}(A)=\alpha(A)$. 

The only thing left to show is that
\begin{equation}
\bigcap_{(s,f) \in \mathcal{P}} \mathcal{A}^{s,f}  \subseteq \mathcal{A}^\mathrm{uni}.
\end{equation}
So assume $A \in \bigcap_{(s,f) \in \mathcal{P}} \mathcal{A}^{s,f}$. Suppose we have $(s,f) \in \mathcal{P}$ such that $\alpha^{s,f}(A)=L(\log(A))$ and $(s',f') \in \mathcal{P}$ such that $\alpha^{s',f'}(A)=U(\log(A))$. Then we can create a new sequence given by
\begin{equation}
s'' := (s_1,s'_1,s_2,s'_2,...) \;\;\mathrm{and}\;\; f'' := (f_1,f'_1,f_2,f'_2,...).
\end{equation}
Because by assumption $A \in \mathcal{A}^{s'',f''}$, we then have $\alpha^{s,f}(A)=\alpha^{s',f'}(A)$. Hence $A \in \mathcal{A}^\mathrm{uni}$. So it is sufficient to show that we can choose $(s,f)$ and $(s',f')$ in the desired way.

Choose $s \in \mathcal{S}$ such that
\begin{equation}
\lim_{n \rightarrow \infty} \inf_{x \in (1,\infty)} \sigma_{\log(A)}(\log(s_n),\log(x)) = \liminf_{D \rightarrow \infty} \inf_{x \in (1,\infty)} \sigma_{\log(A)}(\log(D),\log(x)).
\end{equation}
Choose $f \in (1,\infty)^\mathbb{N}$ such that
\begin{equation}
\left| \inf_{x \in (1,\infty)} \sigma_{\log(A)}(\log(s_n),\log(x)) - \sigma_{\log(A)}(\log(s_n),\log(f_n))\right| < \frac{1}{n}
\end{equation}
for every $n \in \mathbb{N}$. Then $(s,f) \in \mathcal{P}$ with
\begin{equation}
\alpha^{s,f}(A) = \liminf_{D \rightarrow \infty} \inf_{x \in (1,\infty)} \sigma_{\log(A)}(\log(D),\log(x)) = L(\log(A)).
\end{equation}
In the same way choose $(s',f') \in \mathcal{P}$ such that
\begin{equation}
\alpha^{s',f'}(A) = \limsup_{D \rightarrow \infty} \sup_{x \in (1,\infty)} \sigma_{\log(A)}(\log(D),\log(x)) = U(\log(A)).
\end{equation}
\qed
\end{proof}

\begin{proof} \textit{of Proposition \ref{prop:alphaextendslambda}} \;
Let $A \in \mathcal{C}$ and $(s,f) \in \mathcal{P}$. Since $\rho_A(y) \rightarrow \lambda(A)$, we have $\xi_A(s_n,f_n) \sim \lambda(A)$, so $\alpha^{s,f}(A)=\lambda(A)$. The result now follows by Lemma \ref{lem:rep1}.
\qed
\end{proof}

\begin{proof} \textit{of Theorem \ref{thm:wtp}} \;
Notice that any intersection of $f$-systems closed under weak thinning is again closed under weak thinning. Therefore, if we show that $(\mathcal{A}^{s,f},\alpha^{s,f})$ is a WTP for every $(s,f)\in\mathcal{P}$, it follows from 
Lemma \ref{lem:rep1} that $(\mathcal{A}^\mathrm{uni},\alpha)$ is a WTP. 

Let $(s,f) \in \mathcal{P}$. It immediately follows that $(\mathcal{A}^{s,f},\alpha^{s,f})$ is a probability pair and that P2 and P3 hold, so we have to verify P1. Note that for every $A,B \in \mathcal{M}$ and $x>0$ we have
\begin{equation}
\begin{aligned}
\rho_{A \circ B}(x) & = \frac{1}{x} \int_0^x 1_{A \circ B}(y) \mathrm{d}y \\
& = \frac{1}{x} \int_0^x 1_A(y)1_B(S_A(y)) \mathrm{d}y \\
& = \frac{1}{x} \int_0^{S_A(x)} 1_B(u) \mathrm{d}u \\
& = \frac{S_A(x)}{x} \frac{1}{S_A(x)}\int_0^{S_A(x)} 1_B(u) \mathrm{d}u \\
& = \rho_A(x) \rho_B(S_A(x)) = \rho_A(x) \rho_B(x\rho_A(x)).
\end{aligned}
\end{equation}

Let $A \in \mathcal{C}$ and $B \in \mathcal{A}^{s,f}$. Then
\begin{equation}
\begin{aligned}
\xi_{A \circ B}(s_n,f_n) & = \frac{1}{\log(s_n)} \int_{f_n}^{s_nf_n} \rho_A(y) \frac{\rho_B(y \rho_A(y))}{y} \mathrm{d}y \\
& \sim  \lambda(A) \frac{1}{\log(s_n)} \int_{f_n}^{s_nf_n} \frac{\rho_B(\lambda(A) y)}{y} \mathrm{d}y.
\end{aligned}
\end{equation}
If $\lambda(A)=0$ it is clear that $A \circ B \in \mathcal{A}^{s,f}$ with $\alpha^{s,f}(A \circ B) = 0 = \lambda(A)\alpha^{s,f}(B)$. If $\lambda(A)>0$, then we see that
\begin{equation}
\int_{f_n}^{s_nf_n} \frac{\rho_B(\lambda(A) y)}{y} \mathrm{d}y = \int_{\lambda(A)f_n}^{\lambda(A)s_nf_n} \frac{\rho_B(u)}{u} \mathrm{d}u.
\end{equation}
Since
\begin{equation}
\begin{aligned}
\left| \int_{\lambda(A)f_n}^{\lambda(A)s_nf_n} \frac{\rho_B(u)}{u} \mathrm{d}u - \int_{f_n}^{s_nf_n} \frac{\rho_B(u)}{u} \mathrm{d}u \right| & \leq \int_{\lambda(A)s_nf_n}^{s_nf_n} \frac{1}{u} \mathrm{d}u + \int_{\lambda(A)f_n}^{f_n} \frac{1}{u} \mathrm{d}u \\
& = 2 \log\left( \frac{1}{\lambda(A)} \right),
\end{aligned}
\end{equation}
we have
\begin{equation}
\label{eq:wtpasym}
\begin{aligned}
\frac{1}{\log(s_n)} \int_{f_n}^{s_nf_n} \frac{\rho_B(\lambda(A) y)}{y} \mathrm{d}y & \sim \frac{1}{\log(s_n)} \int_{f_n}^{s_nf_n} \frac{\rho_B(u)}{u} \mathrm{d}u \sim \alpha^{s,f}(B).
\end{aligned}
\end{equation}
Thus $A \circ B \in \mathcal{A}^{s,f}$ and since $\lambda(A)=\alpha^{s,f}(A)$ (see the proof of Propositon \ref{prop:alphaextendslambda}), we have
\begin{equation}
\alpha^{s,f}(A \circ B) = \lambda(A)\alpha^{s,f}(B) = \alpha^{s,f}(A) \alpha^{s,f}(B).
\end{equation}

We have showed that $(\mathcal{A}^\mathrm{uni},\alpha) \in WT$. To show that $(\mathcal{A}^\mathrm{uni},\alpha)$ can be extended, let $\mathcal{U}$ be any free ultrafilter on $\mathbb{N}$ and $(s,f) \in \mathcal{P}$. Then consider $\mu: \mathcal{M} \rightarrow [0,1]$ given by
\begin{equation}
\mu(A) := \mathcal{U} \mhyphen \lim_{n \rightarrow \infty} \xi_A(s_n,f_n).
\end{equation}
Since the $\mathcal{U}$-limit is multiplicative it follows completely analogous that $(\mathcal{M},\mu)$ is a WTP. Hence every $(\mathcal{A}^{s,f},\alpha^{s,f})$ can be extended to a WTP with $\mathcal{M}$ as its $f$-system. In particular, by Lemma \ref{lem:rep1}, this means that $(\mathcal{A}^{uni},\alpha)$ can be extended to a WTP with $\mathcal{M}$ as its $f$-system.

From de Finetti \cite{definetticoherence} it follows that if $\alpha$ can be extended to a finitely additive probability measure on an algebra, then $\alpha$ is coherent. Since we have showed that $\alpha$ can be extended to $\mathcal{M}$, which is an algebra, it follows that $(\mathcal{A}^\mathrm{uni},\alpha) \in WTC$. Notice that we showed that $\alpha^{s,f}$ can be extended to $\mathcal{M}$ for every $(s,f) \in \mathcal{P}$, so we also have $(\mathcal{A}^{s,f},\alpha^{s,f}) \in WTC$ for every $(s,f) \in \mathcal{P}$.

\qed
\end{proof}

For our proof of Theorem \ref{thm:unique}, we need an alternate expression for $U(\log(A))$. For $A \in \mathcal{M}$ set $\tau_A: (1,\infty) \times \mathbb{N} \rightarrow [0,1]$ given by
\begin{eqnarray}
\tau_A(C,j) & := & \sigma_A( C^{j-1}(C-1),C^{j-1}) \\
& = & \frac{1}{C^{j-1}(C-1)} \int_{C^{j-1}}^{C^j} 1_A(y) \mathrm{d}y.
\end{eqnarray}
Also set for $C>1$ and $A \in \mathcal{M}$
\begin{equation}
U^*(C,A) := \limsup_{n \rightarrow \infty} \sup_{k \in \mathbb{N}} \frac{1}{n} \sum_{j=k}^{k+n-1} \tau_A(C,j).
\end{equation}

\begin{lemma}
\label{lem:UlogA}
For every $A \in \mathcal{M}$ we have
\begin{equation}
\lim_{C \downarrow 1} U^*(C,A) = U(\log(A)).
\end{equation}
\end{lemma}

\begin{proof} 
Let $A \in \mathcal{M}$ and fix $C>1$.

\textbf{Step 1} We show that 
\begin{equation}
U(\log(A))  = \limsup_{D \rightarrow \infty} \sup_{x \in (0,\infty)} \frac{1}{D} \sum_{j=P(x)+1}^{Q(D,x)} \int_{C^{j-1}}^{C^{j}} \frac{1_A(u)}{u} \mathrm{d}u,
\end{equation}
where 
\begin{equation}
P(x):= \left\lceil \frac{x}{\log(C)} \right\rceil \;\;\mathrm{and}\;\; Q(D,x):= \left\lceil \frac{D+x}{\log(C)} \right\rceil
\end{equation}
for $D,x \in (0,\infty)$.

Define
\begin{equation}
E(D,x):= \sigma_{\log(A)}(D,x) - \frac{1}{D} \int_{C^{P(x)}}^{C^{Q(D,x)}} \frac{1_A(u)}{u} \mathrm{d}u.
\end{equation}
Since
\begin{equation}
\sigma_{\log(A)}(D,x) = \frac{1}{D} \int_{x}^{x+D} 1_A(e^y) \mathrm{d}y = \frac{1}{D} \int_{e^x}^{e^{x+D}} \frac{1_A(u)}{u} \mathrm{d}u,
\end{equation}
we have
\begin{equation}
|E(D,x)| \leq \frac{1}{D} \int_{C^{P(x)-1}}^{C^{P(x)}} \frac{1}{u} \mathrm{d}u + \frac{1}{D} \int_{C^{Q(D,x)}}^{C^{Q(D,x)+1}} \frac{1}{u} \mathrm{d}u =\frac{2}{D} \log(C).
\end{equation}
This implies
\begin{equation}
\label{eq:startexpression}
\begin{aligned}
U(\log(A)) & = \limsup_{D \rightarrow \infty} \sup_{x \in (0,\infty)} \sigma_{\log(A)}(D,x) \\
& = \limsup_{D \rightarrow \infty} \sup_{x \in (0,\infty)}  \frac{1}{D} \int_{C^{P(x)}}^{C^{Q(D,x)}} \frac{1_A(u)}{u} \mathrm{d}u \\
& = \limsup_{D \rightarrow \infty} \sup_{x \in (0,\infty)} \frac{1}{D} \sum_{j=P(x)+1}^{Q(D,x)} \int_{C^{j-1}}^{C^{j}} \frac{1_A(u)}{u} \mathrm{d}u.
\end{aligned}
\end{equation}

\textbf{Step 2} We give an upper and lower bound for 
\begin{equation}
\int_{C^{j-1}}^{C^{j}} \frac{1_A(u)}{u} \mathrm{d}u
\end{equation}
in terms of $\tau_A(C,j)$.

If we set for $j \in \mathbb{N}$
\begin{equation}
\zeta(j) := \int_{C^{j-1}}^{C^j} 1_A(y) \mathrm{d}y = \tau_A(C,j)(C-1)C^{j-1},
\end{equation}
then
\begin{equation}
\label{eq:upperandlower1}
\int_{C^j-\zeta(j)}^{C^j} \frac{1}{u}\mathrm{d}u \leq \int_{C^{j-1}}^{C^j} \frac{1_A}{u}\mathrm{d}u \leq \int_{C^{j-1}}^{C^{j-1}+\zeta(j)} \frac{1}{u}\mathrm{d}u.
\end{equation}
We now observe that
\begin{equation}
\label{eq:computeupper}
\begin{aligned}
\int_{C^{j-1}}^{C^{j-1}+\zeta(j)} \frac{1}{u}\mathrm{d}u & =  \log\left(\frac{C^{j-1}+\zeta(j)}{C^{j-1}}\right) \\
& =  \log(1 + (C-1)\tau_A(C,j))
\end{aligned}
\end{equation}
and
\begin{equation}
\label{eq:computeunder}
\begin{aligned}
\int_{C^j-\zeta(j)}^{C^j} \frac{1}{u}\mathrm{d}u & =  \log\left(\frac{C^j}{C^j-\zeta(j)}\right) \\
& = \log(C) - \log\left( 1 + (C-1) (1 - \tau_A(C,j)) \right).
\end{aligned}
\end{equation}
The fact that $\log(1+y) \leq y$ for every $y\geq 0$, combined with (\ref{eq:upperandlower1}), (\ref{eq:computeupper}) and (\ref{eq:computeunder}) gives
\begin{equation}
\label{eq:upperandlower2}
\log(C) - (C-1) (1 - \tau_A(C,j))  \leq \int_{C^{j-1}}^{C^j} \frac{1_A}{u}\mathrm{d}u \leq (C-1)\tau_A(C,j).
\end{equation}

\textbf{Step 3} We combine Step 1 and Step 2 to finish the proof.

Observe that
\begin{equation}
\label{eq:linkdiscrete}
\begin{aligned}
& \limsup_{D \rightarrow \infty} \sup_{x \in(0,\infty)} \frac{1}{Q(D,x)-P(x)} \sum_{j=P(x)+1}^{Q(D,x)} \tau_A(C,j) \\  = &   \limsup_{n \rightarrow \infty} \sup_{k \in \mathbb{N}} \frac{1}{n} \sum_{j=k}^{k+n-1} \tau_A(C,j) = U^*(C,A).
\end{aligned}
\end{equation}

We use (\ref{eq:upperandlower2}) and (\ref{eq:linkdiscrete}) to find an upperbound for the expression in (\ref{eq:startexpression}), giving us
\begin{equation}
\label{eq:upperforU}
\begin{aligned}
U(\log(A)) & =  \limsup_{D \rightarrow \infty} \sup_{x \in (0,\infty)} \frac{1}{D} \sum_{j=P(x)+1}^{Q(D,x)} \int_{C^{j-1}}^{C^{j}} \frac{1_A(u)}{u} \mathrm{d}u  \\
& \leq \limsup_{D \rightarrow \infty} \sup_{x \in (0,\infty)} \frac{C-1}{D} (Q(D,x)-P(x)) \gamma(D,x) \\
& = \frac{C-1}{\log(C)} \limsup_{D \rightarrow \infty} \sup_{x \in (0,\infty)} \frac{1}{Q(D,x)-P(x)} \sum_{j=P(x)+1}^{Q(D,x)} \tau_A(C,j) \\
& = \frac{C-1}{\log(C)} U^*(C,A).
\end{aligned}
\end{equation}

Analogously, we find that
\begin{equation}
\label{eq:lowerforU}
\begin{aligned}
U(\log(A)) & \geq 1 - \frac{C-1}{\log(C)} (1 - U^*(C,A)).
\end{aligned}
\end{equation}

Combining (\ref{eq:upperforU}) and (\ref{eq:lowerforU}) we obtain
\begin{equation}
\frac{\log(C)}{C-1} U(\log(A)) \leq U^*(C,A) \leq 1-\frac{\log(C)}{C-1}(1-U(\log(A))),
\end{equation}
which implies
\begin{equation}
\label{eq:identityforU}
\lim_{C \downarrow 1} U^*(C,A) = U(\log(A)).
\end{equation}
\qed
\end{proof}

We also need the following lemma.

\begin{lemma}
\label{lem:ineqforB}
Let $(\mathcal{F},\mu)$ be a WTP. Then for any $A \in \mathcal{F}$ and $C>1$
\begin{equation}
\mu(A) \leq C \sup_{j \in \mathbb{N}} \tau_A(C,j).
\end{equation}
\end{lemma}

\begin{proof}
Let $(\mathcal{F},\mu)$ be a WTP with $A \in \mathcal{F}$. Fix $C>1$ and write
\begin{equation}
S := \sup_{j \in \mathbb{N}} \tau_A(C,j).
\end{equation}
The idea is to introduce a set $B \in \mathcal{M}$ for which we have $\limsup_{x \rightarrow \infty} \rho_B(x) \leq CS$ and $\rho_A(x) \leq \rho_B(x)$ for all $x$. Set
\begin{equation}
B := \bigcup_{j=1}^\infty [C^{j-1}, C^{j-1} + S C^{j-1}(C-1)).
\end{equation}
By construction of $B$ we have $\rho_A(x) \leq \rho_B(x)$ for every $x \in (0,\infty)$. So
\begin{equation}
\begin{aligned}
\limsup_{x \rightarrow \infty} \rho_A(x) & \leq  \limsup_{x \rightarrow \infty} \rho_B(x) \\
& =  \limsup_{n \rightarrow \infty} \rho_B(C^n + S C^n (C-1)) \\
& =  \limsup_{n \rightarrow \infty} \sum_{j=1}^{n+1} \frac{S C^{j-1}(C-1)} {C^n + S C^n (C-1)} \\
& =  \limsup_{n \rightarrow \infty} \frac{S (C^{n+1}-1)} {C^n + S C^n (C-1)} \\
& =  \limsup_{n \rightarrow \infty} \frac{C S} {\frac{C}{C-C^{-n}}(1 + S(C-1))} \\
& =  \frac{C S} {1 + S(C-1)} \leq CS.
\end{aligned}
\end{equation}
By Proposition \ref{prop:muandrho} we then find
\begin{equation}
\mu(A) \leq \limsup_{x \rightarrow \infty} \rho_A(x) \leq C S.
\end{equation}
\qed
\end{proof}

We are ready to give the proof of Theorem \ref{thm:unique}.

\begin{proof} \textit{of Theorem \ref{thm:unique}} \;
Let $(\mathcal{F},\mu)$ be a WTP and $A \in \mathcal{F}$. It is sufficient to show that
\begin{equation}
L(\log(A)) \leq \mu(A) \leq U(\log(A)).
\end{equation}

We give the following example to give an idea of the proof that follows. Set
\begin{equation*}
\begin{aligned}
Z_1 & :=  \bigcup_{i=1}^\infty [2i,2i+1) = [2,3) \cup [4,5) \cup [6,7) \cup \ldots, \\
Z_2 & :=  \bigcup_{i=1}^\infty [4i+1,4i+2) = [5,6) \cup [9,10) \cup [13,14) \cup \ldots, \\
Z_3 & :=  \bigcup_{i=1}^\infty [4i+3,4i+4) = [7,8) \cup [11,12) \cup [15,16) \cup \ldots.
\end{aligned}
\end{equation*}
Note that $Z_1,Z_2,Z_3 \in \mathcal{C}$ are pairwise disjoint. Now, we set
\begin{equation}
A' :=  Z_1 \circ A + Z_2 \circ A + Z_3 \circ A.
\end{equation}
Observe that for $j \geq 3$
\begin{equation}
\tau_{A'}(2,j) = \frac{1}{2} \left( \tau_A(2,j-1) + \tau_A(2,j-2) \right).
\end{equation}
So we constructed a set $A'$ that on each interval $[2^{j-1},2^{j})$ with $j \geq 3$ has an average that equals the average of the averages of $A$ on two consecutive intervals. By weak thinnability we find that $\mu(A')=\frac{1}{2}\mu(A) + \frac{1}{4}\mu(A) + \frac{1}{4}\mu(A) = \mu(A)$. If $\tau_{A'}(2,j)$ is convergent or only oscillates a little, we can give a good upper bound of $\mu(A)$ using Lemma \ref{lem:ineqforB}. Applying this strategy not only for $C=2$ but for any $C>1$ and averages of not only two but arbitrarily many averages on consecutive intervals, is what happens in the proof.

\textbf{Step 1} We construct a $\hat{A} \in \mathcal{F}$.

Fix $C>1$ and $n \in \mathbb{N}$. We split up $[C^{j-1},C^j)$ into intervals of length $1$ plus a remainder interval for every $j$. Set for $j \in \mathbb{N}$
\begin{equation}
N_j := \left\lfloor C^{j-1}(C-1) \right\rfloor
\end{equation}
and for $j \in \mathbb{N}$ and $l \in \{1,...,N_j\}$
\begin{equation}
I(j,l) := \left[ C^{j-1} + l-1, C^{j-1} + l \right),
\end{equation}
so that for every $j \in \mathbb{N}$ we have
\begin{equation}
[C^{j-1},C^j) =  \left[ C^{j-1} + N_j, C^{j} \right) \cup \bigcup_{l=1}^{N_j} I(j,l).
\end{equation}

Choose $u \in \mathbb{N}$ such that for every $j \in \mathbb{N}$ we have
\begin{equation}
\label{eq:ubigenough}
N_{u+j} \geq (N_{n+j}+1) \sum_{p=0}^n \lfloor C^{p} \rfloor,
\end{equation}
which can be done since $N_j$ is asymptotically equivalent with $C^{j-1}(C-1)$. For $p \in \{0,..,n\}$, $k \in \{1,..,\lfloor C^{p} \rfloor\}$ and $j \in \mathbb{N}$ we set 
\begin{equation}
I^{p,k}(j) :=  \bigcup_{l} I\left(j, l\sum_{i=0}^n \lfloor C^{i} \rfloor + \sum_{i=0}^{p-1} \lfloor C^{i} \rfloor  + k \right).
\end{equation}
For $l \leq T$ set
\begin{equation}
\zeta(l,T) := \bigcup_{i=0}^{\lceil l \rceil -1} \left[ \frac{Ti}{\lceil l \rceil}, \frac{Ti + l}{\lceil l \rceil} \right)
\end{equation}
that `evenly' distributes mass $l$ over the interval $[0,T)$. Note that (\ref{eq:ubigenough}) guarantees that
\begin{equation}
m(I^{p,k}(u+j)) \geq C^{n+j-1}(C-1) \geq C^{n-p+j-1}(C-1)
\end{equation}
for every $j \in \mathbb{N}$, so
\begin{equation}
Z(p,k) := \bigcup_{j=1}^\infty \left( I^{p,k}(u+j) \circ \zeta\left( C^{n-p+j-1}(C-1),  m (I^{p,k}(u+j)) \right) \right)
\end{equation}
is well defined. Note that by construction $Z(p,k) \in \mathcal{C}$ and
\begin{equation}
m (Z(p,k) \cap I^{p,k}(u+j)) = C^{n-p+j-1}(C-1).
\end{equation}
From this it directly follows that
\begin{equation}
\lambda(Z(p,k)) = \frac{C^n}{C^{p+u}}.
\end{equation}

We now introduce
\begin{equation}
\hat{A} := \bigcup_{p=0}^n \bigcup_{k=1}^{\lfloor C^p \rfloor} Z(p,k) \circ A.
\end{equation}
Observe that all the $Z(p,k)$ are disjoint. So P1 and the fact that $\mathcal{F}$ is an $f$-system imply that $\hat{A} \in \mathcal{F}$.

\textbf{Step 2} We give an upperbound for $\mu(A)$ by first giving an upperbound for $\mu(\hat{A})$ and then relating $\mu(A)$ and $\mu(\hat{A})$. 

A crucial property of $\hat{A}$ is that for $j \in \mathbb{N}$
\begin{equation}
m ( [C^{u+j-1},C^{u+j}) \cap \hat{A}) = \sum_{p=0}^n \lfloor C^p \rfloor m( [C^{j+n-p-1},C^{j+n-p}) \cap A ). 
\end{equation}
Hence
\begin{equation}
\label{eq:upperfortau}
\begin{aligned}
\tau_{\hat{A}}(C,u+j) & = C^{n-u} \sum_{p=0}^n \lfloor C^p \rfloor C^{-p}\tau_A(C,j+n-p) \\
& \leq  C^{n-u} \sum_{p=0}^n \tau_A(C,j+n-p) \\
& \leq  C^{n-u} \sup_{k \in \mathbb{N}} \sum_{j=k}^{k+n} \tau_A(C,j).
\end{aligned}
\end{equation}
We apply Lemma \ref{lem:ineqforB} for $\hat{A}$ and find with (\ref{eq:upperfortau}) that
\begin{equation}
\label{eq:upperformuhat}
\mu(\hat{A}) \leq C^{n-u+1} \sup_{k \in \mathbb{N}} \sum_{j=k}^{k+n} \tau_A(C,j).
\end{equation} 

The weak thinnability of $\mu$ gives that
\begin{equation}
\label{eq:relatexandxhat}
\mu(\hat{A}) = \sum_{p=0}^n \sum_{k=1}^{a_p} \mu(Z(p,k))\mu(A) = \mu(A) C^{n-u} \sum_{p=0}^n \lfloor C^p \rfloor C^{-p} .
\end{equation}
Combining (\ref{eq:upperformuhat}) and (\ref{eq:relatexandxhat}) gives
\begin{equation}
\label{eq:upperforA}
\begin{aligned}
\mu(A) & = \frac{C^{u-n}}{\sum_{p=0}^n \lfloor C^p \rfloor C^{-p}} \mu(\hat{A}) \\
& \leq  \frac{C^{u-n}}{\sum_{p=0}^n (C^p-1) C^{-p}} \mu(\hat{A}) \\
& \leq  \frac{C^{u-n}}{n+1 - \frac{1}{1-1/C}} \mu(\hat{A}) \\
& \leq  C \frac{n+1}{n+1 - \frac{1}{1-1/C}} \sup_{k \in \mathbb{N}} \frac{1}{n+1} \sum_{j=k}^{k+n} \tau_A(C,j).
\end{aligned}
\end{equation}

\textbf{Step 3} We take limits in (\ref{eq:upperforA}).

Unfix $n$ and $C$. We first take the limit superior for $n \rightarrow \infty$ in (\ref{eq:upperforA}), giving
\begin{equation}
\mu(A) \leq C \limsup_{n \rightarrow \infty} \sup_{k \in \mathbb{N}} \frac{1}{n+1} \sum_{j=k}^{k+n} \tau_A(C,j) = C U^*(C,A).
\end{equation}
Then we take the limit superior for $C \downarrow 1$ and find by Lemma \ref{lem:UlogA} that
\begin{equation}
\mu(A) \leq \limsup_{C \downarrow 1} U^*(C,A) = U(\log(A)).
\end{equation}
The lower bound we can now easily obtain by applying our upper bound for the complement of $A$. Doing this, we see that
\begin{equation}
\begin{aligned}
1-\mu(A) & \;=\;  \mu(A^c) \\
& \;\leq\;  U(\log(A^c)) \\
& \;=\;  1 - L(\log(A)),
\end{aligned}
\end{equation}
giving that $\mu(A) \geq L(\log(A))$.
\qed
\end{proof}

\begin{proof} \textit{of Theorem \ref{thm:maximal}}
We prove the contrapositive. Let $(\mathcal{F},\mu)$ be a WTP with $\mathcal{F} \setminus \mathcal{A}^\mathrm{uni} \not=\emptyset$. Let $A \in \mathcal{F} \setminus \mathcal{A}^\mathrm{uni}$. By Lemma \ref{lem:rep1}, this means that there is a $(s,f) \in \mathcal{P}$ such that
\begin{equation*}
I := \liminf_{n \rightarrow \infty} \xi_A(s_n,f_n) \not = \limsup_{n \rightarrow \infty} \xi_A(s_n,f_n) =: S.
\end{equation*}
Clearly, we can find $m,l \in \mathbb{N}^\infty$ such that $\xi_A(s_{m_n},f_{m_n})$ tends to $I$ and $\xi_A(s_{l_n},f_{l_n})$ tends to $S$. Now set $s'_n := s_{m_n}$, $f'_n := f_{m_n}$, $s''_n := s_{l_n}$ and $f''_n := f_{l_n}$. Then we see that $A \in \mathcal{A}^{s',f'}$ and $A \in \mathcal{A}^{s'',f''}$ with
\begin{equation*}
\alpha^{s',f'}(A) = I \;\;\mathrm{and}\;\; \alpha^{s'',f''}(A) = S.
\end{equation*}
In the proof of Theorem \ref{thm:wtp} we showed that $(\mathcal{A}^{s',f'},\alpha^{s',f'})$ and $(\mathcal{A}^{s'',f''},\alpha^{s'',f''})$ are both in WTC. Thus $(\mathcal{F},\mu)$ is not canonical with respect to $WT$ and in case $\mu$ is coherent, $(\mathcal{F},\mu)$ is not canonical with respect to $WTC$.

\qed
\end{proof}

\begin{proof} \textit{of Proposition \ref{prop:uniformunique}} \;
We give a proof along the lines of Mattila \cite[p.~45]{mattila}, with small adaptations for completeness and more generality.

Let $(X,d)$ be a metric space and $\nu_1,\nu_2$ uniform measures on $X$. Write $h_1(r):= \nu_1(B(x,r))$ and $h_2(r):= \nu_2(B(x,r))$ for $r>0$, which are well defined since $\nu_1$ and $\nu_2$ are uniform. We show that $\nu_1=c\nu_2$ for some $c>0$.  It is sufficient to show that $\nu_1=c\nu_2$ on all open sets.

First let $A$ be an open set of $(X,d)$ with $\nu_1(A)<\infty$ and $\nu_2(A)<\infty$. Suppose that $r>0$ is such that $h_2$ is continuous in $r$. Then
\begin{equation}
\begin{aligned}
|\nu_2(A \cap B(x,r))- \nu_2(A \cap B(y,r))| & \leq \nu_2(B(x,r) \triangle B(y,r)) \\
& \leq \nu_2( B(x,r+d(x,y)) \setminus B(x,r) ) \\
& = h_2(r+d(x,y)) - h_2(r).
\end{aligned}
\end{equation}
Hence $x \mapsto \nu_2(A \cap B(x,r))$ is a continuous mapping from $X$ to $[0,\infty)$. Since $h_2$ is nondecreasing, it can have at most countable many discontinuities. So we can choose $r_1,r_2,r_3,...$ such that $\lim_{n \rightarrow \infty} r_n = 0$ and $h_2$ is continuous in every $r_n$. 

For $n \in \mathbb{N}$ let $f_n: X \rightarrow [0,1]$ be given by
\begin{equation}
f_n(x) := 1_A(x) \frac{\nu_2(A \cap B(x,r_n))}{h_2(r_n)}.
\end{equation}
Notice that by our previous observation $f_n$ is continuous on $A$, hence $f_n$ is measurable. Because $A$ is open, we have $\lim_{n \rightarrow \infty} f_n(x) = 1$ for every $x \in A$. With Fatou's Lemma we find
\begin{equation}
\begin{aligned}
\nu_1(A) & =  \int_A \lim_{n \rightarrow \infty} f_n(x) \nu_1(\mathrm{d}x) \\
& \leq  \liminf_{n \rightarrow \infty} \frac{1}{h_2(r_n)} \int_A \nu_2(A \cap B(x,r_n)) \nu_1(\mathrm{d}x) \\
& \leq  \liminf_{n \rightarrow \infty} \frac{1}{h_2(r_n)} \int_X \int_A 1_{B(x,r_n)}(y) \nu_2(\mathrm{d}y) \nu_1(\mathrm{d}x).
\end{aligned}
\end{equation}
Note that any uniform measure is $\sigma$-finite. Applying Fubini's theorem we obtain
\begin{equation}
\begin{aligned}
\nu_1(A) & \leq \liminf_{n \rightarrow \infty} \frac{1}{h_2(r_n)} \int_A \int_X 1_{B(x,r_n)}(y) \nu_1(\mathrm{d}x) \nu_2(\mathrm{d}y) \\
& = \liminf_{n \rightarrow \infty} \frac{1}{h_2(r_n)} \int_A \nu_1(B(y,r_n)) \nu_2(\mathrm{d}y) \\
& = \liminf_{n \rightarrow \infty} \frac{h_1(r_n)}{h_2(r_n)} \nu_2(A).
\end{aligned}
\end{equation}
By interchanging $\nu_1$ and $\nu_2$ we get
\begin{equation}
\label{eq:interchanged}
\nu_2(A) \leq \liminf_{n \rightarrow \infty} \frac{h_2(r_n)}{h_1(r_n)} \nu_1(A).
\end{equation}
Note that $\liminf_{n \rightarrow \infty} \frac{h_2(r_n)}{h_1(r_n)}>0$ since (\ref{eq:interchanged}) would otherwise imply that all open balls are null sets. So we may rewrite (\ref{eq:interchanged}) as
\begin{equation}
\begin{aligned}
v_1(A) & \geq \frac{1}{\liminf_{n \rightarrow \infty} \frac{h_2(r_n)}{h_1(r_n)}} \nu_2(A) \\
& = \limsup_{n \rightarrow \infty} \frac{h_1(r_n)}{h_2(r_n)} \nu_2(A) \\
& \geq \liminf_{n \rightarrow \infty} \frac{h_1(r_n)}{h_2(r_n)} \nu_2(A).
\end{aligned}
\end{equation}
Hence $v_1(A)=cv_2(A)$ with
\begin{equation}
c: =  \liminf_{n \rightarrow \infty} \frac{h_1(r_n)}{h_2(r_n)} > 0.
\end{equation}

Now let $A$ be any open set of $(X,d)$. Let $x \in X$ and set $A_n := A \cap B(x,n)$ for $n \in \mathbb{N}$. Note that $A_n$ is open with $\nu_1(A_n)\leq \nu_1(B(x,n))<\infty$ and $\nu_2(A_n) \leq \nu_2(B(x,n))<\infty$. Hence, by the first part of the proof, we find $\nu_1(A_n)=c\nu_2(A_n)$. But then
\begin{equation}
\nu_1(A) = \lim_{n \rightarrow \infty} \nu_1(A_n) = \lim_{n \rightarrow \infty} c\nu_2(A_n) = c \nu_2(A).
\end{equation}
\qed
\end{proof}

\begin{proof} \textit{of Proposition \ref{prop:asymptotic}} \;
Fix $A \in \mathcal{L}(X)$ and $x,y \in X$. By (\ref{eq:assumptionX}) we have
\begin{equation}
\lim_{u \rightarrow \infty} \frac{h(r^+(u))}{h(r^-(u))}=\lim_{u \rightarrow \infty} \frac{h(r^-(u)+1)}{h(r^-(u))}=\lim_{r \rightarrow \infty} \frac{h(r+1)}{h(r)}=1.
\end{equation}
Hence
\begin{equation}
\label{eq:asympt1}
\frac{\bar{\nu}(B(x,r^-(u))\cap A)}{h(r^-(u))} \sim \frac{\bar{\nu}(B(x,r^+(u))\cap A)}{h(r^+(u))}.
\end{equation}
Observe that for any $r \in [0,\infty)$ we have
\begin{equation}
\begin{aligned}
\left| \frac{\bar{\nu}(B(x,r)\cap A)}{h(r)} - \frac{\bar{\nu}(B(y,r)\cap A)}{h(r)} \right|
& = \frac{1}{h(r)} |\bar{\nu}(A \cap B(x,r))- \bar{\nu}(A \cap B(y,r))| \\
&  \leq \frac{1}{h(r)} \nu(B(x,r) \triangle B(y,r)) \\
&  \leq \frac{1}{h(r)} \nu( B(x,r+d(x,y)) \setminus B(y,r) ) \\
& = \frac{h(r+d(x,y)) - h(r)}{h(r)}.
\end{aligned}
\end{equation}
By (\ref{eq:assumptionX}), it follows that
\begin{equation}
\label{eq:asympt2}
\frac{\bar{\nu}(B(x,r^-(u))\cap A)}{h(r^-(u))} \sim \frac{\bar{\nu}(B(y,r^-(u))\cap A)}{h(r^-(u))} 
\end{equation}
Combining (\ref{eq:asympt1}) and (\ref{eq:asympt2}) gives the desired result.
\qed
\end{proof}

\begin{proof} \textit{of Proposition \ref{prop:euclideanspace}} \;
Suppose $X=\mathbb{R}^n$ with $d$ Euclidean distance. Set
\begin{equation}
\delta_n := \frac{2 \pi^{n/2}}{\Gamma(n/2)}.
\end{equation}
Let $\nu$ be the Borel measure on $\mathbb{R}^n$. Note that $h(r) = n^{-1} \delta_n r^n$. If we set $u=\sqrt[n]{n\delta_n^{-1}y}$, then  
\begin{equation}
\begin{aligned}
\int_x^{xD} \frac{\bar{\rho}_A(y)}{y} \mathrm{d}y & = \int_x^{xD} \frac{\delta_n}{y^2} \int_0^{\sqrt[n]{n\delta_n^{-1}y}} r^{n-1} K_A(r) \mathrm{d}r \mathrm{d}y \\
& = \int_{ \sqrt[n]{n\delta_n^{-1}x} }^{ \sqrt[n]{n\delta_n^{-1}xD} } \frac{n^2}{u^{n+1} } \int_0^{u} r^{n-1} K_A(r) \mathrm{d}r \mathrm{d}u.
\end{aligned}
\end{equation}
Now observe that by partial integration
\begin{equation}
\int \frac{n^2}{u^{n+1}} \int_0^{u} r^{n-1} K_A(r) \mathrm{d}r \mathrm{d}u = \frac{-n}{u^n} \int_0^{u} r^{n-1} K_A(r) \mathrm{d}r + n \int \frac{K_A(u)}{u} \mathrm{d}u.
\end{equation}
If we set for $D,x \in (1,\infty)$
\begin{equation}
\zeta_A(D,x) := - \frac{1}{\log(D) u^n} \int_0^{u} r^{n-1} K_A(r) \mathrm{d}r \biggr\rvert_{u=x}^{xD},
\end{equation}
then
\begin{equation}
\bar{\xi}_A(D^n,n^{-1}\delta_n x^n) = \zeta_A(D,x) + \frac{1}{\log(D)} \int_x^{xD} \frac{K_A(u)}{u} \mathrm{d}u.
\end{equation}
Since $|\zeta_A(D,x)| \leq \frac{1}{\log(D)}$, the desired result follows.
\end{proof}

\section{Discussion}
\label{sec:discussion}

\subsection{Algebra versus $f$-system}

The natural analogue of an $\sigma$-algebra in finite additive probability theory is an algebra. It has been remarked \cite{schurzleitgeb,wenmackershorsten} that the restriction of $\mathcal{M}$ to $\mathcal{C}$ is problematic since $\mathcal{C}$ is not an algebra. However, \emph{any} collection extending $\mathcal{C}$ that is not $\mathcal{M}$ itself, is not an algebra since $a(\mathcal{C})=\mathcal{M}$. This can be seen as follows. Let $A \in \mathcal{M}$ and set 
\begin{equation}
\begin{aligned}
A_+ & := \{ a+1 \;:\; a \in A \} \\
A_- & := \{ a-1 \;:\; a \in A\setminus[0,1) \} \\
M_1 & := \cup_{i=0}^\infty [2i,2i+1), \\
M_2 & := M_1^c, \\
X & :=  (A \cap M_1) \cup (A_+^c \cap M_2), \\
Y & :=  (A \cap M_2) \cup (A_-^c \cap M_1).
\end{aligned}
\end{equation}
Then $M_1,M_2,X,Y \in \mathcal{C}$ with $\lambda(M_1)=\lambda(M_2)=\lambda(X)=\lambda(Y)=1/2$ and $A = (M_1 \cap X) \cup (M_2 \cap Y)$. Hence $A \in a(\mathcal{C})$ and since $A \in \mathcal{M}$ was arbitrary, we have $a(\mathcal{C})=\mathcal{M}$.

This observation bring us to the conclusion that the requirement of an algebra, despite the fact that an algebra is the natural analogue of an $\sigma$-algebra, is too restrictive. Furthermore, finite additivity \emph{only} dictates how a probability measure behaves when taking \emph{disjoint} unions, and thus only suggests closedness under disjoint unions. Coherence is a concern since, as remarked before, it is not guaranteed on $f$-systems whereas it is guaranteed on algebras. Coherence, however, can also be achieved on $f$-systems, as $\alpha$ does, and therefore coherence not being guaranteed is in itself not an argument against $f$-systems. Therefore, we think the requirement of an $f$-system rather than an algebra in Definition \ref{def:probabilitypair} is justified.

It should be noted that even if one prefers $\mathcal{M}$ as domain, by Theorem \ref{thm:wtp} $(\mathcal{A}^\mathrm{uni},\alpha)$ can be extended to a WTP with $\mathcal{M}$ as $f$-system. Such a pair is not canonical with respect to $WT$ or $WTC$ (Theorem \ref{thm:maximal}), but still has $\mathcal{A}^\mathrm{uni}$ included as an $f$-system within the domain on which probability is uniquely determined. 

\subsection{Thinnability}
\label{subsec:thinnability}

Suppose that in Definition \ref{def:wtp} we replace P1 by the property that for every $A,B \in \mathcal{F}$ we have $A \circ B \in \mathcal{F}$ and $\mu(A \circ B)=\mu(A)\mu(B)$. Instead of \emph{weak} thinnability, we call this \emph{thinnability}. Now consider the set
\begin{equation}
A = \bigcup_{n=0}^\infty [2^{2n},2^{2n+1}).
\end{equation}
We have $A,A^c \in \mathcal{A}^\mathrm{uni}$ with $\alpha(A)=\alpha(A^c)=1/2$. But also, we have $A\circ A^c \in \mathcal{A}^\mathrm{uni}$ with
\begin{equation}
\begin{aligned}
\alpha(A\circ A^c) & = \alpha \left( \bigcup_{n=0}^\infty \left[2^{2n} + \frac{1}{6}2^{2n}, 2^{2n} + \frac{2}{3}2^{2n} \right)\right) \\
& = \lambda \left( \bigcup_{n=0}^\infty \left[2n\log(2) + \log(1+1/6), 2n\log(2) + \log(1+1/3) \right)\right) \\
& = \frac{\log(1+2/3)-\log(1+1/6)} {2\log(2)} \not = \frac{1}{4} = \alpha(A)\alpha(A^c).
\end{aligned}
\end{equation}
So $(\mathcal{A}^\mathrm{uni},\alpha)$ is not a thinnable pair. Since every thinnable pair is also a WTP, by Theorem \ref{thm:unique} we see that a thinnable probability measure on $\mathcal{A}^\mathrm{uni}$, does not exist.

Notice that we are not necessarily looking for the strongest notion of uniformity, but for a notion that allows for a canonical probability pair with a ``big'' $f$-system. This is the reason why we are interested in weak thinnability rather than thinnability. There may, of course, be other notions of uniformity that lead to canonical pairs with bigger $f$-systems than $\mathcal{A}^\mathrm{uni}$. At this point, we can not see any convincing motivation for such notions. 

\subsection{Weak thinnability}

In this paper, we only studied the notion of weak thinnability from the interest in canonical probability pairs. There are, however, interesting open questions about the property of weak thinnability itself, that we did not address in this paper. Some examples are:

\begin{itemize}
\item Is every probability pair that extends $(\mathcal{A}^\mathrm{uni},\alpha)$ a WTP?
\item Is every WTP coherent?
\item Can every WTP be extended to a WTP with $\mathcal{M}$ as $f$-system?
\item How do the sets $\{ \mu(A) \;:\; (\mathcal{F},\mu) \in WT \;\mathrm{and}\; A \in \mathcal{F} \}$ and $\{ \mu(A) \;:\; (\mathcal{F},\mu) \in WTC \;\mathrm{and}\; A \in \mathcal{F} \}$ look like for $A \not\in \mathcal{A}^\mathrm{uni}$?
\item Is P2 redundant? If no, what probability pairs are not a WTP, but do satisfy P1 and P3?
\item How does weak thinnability relate to the property $\mu(cA)=\mu(A)$, where $cA:= \{ ca :\:\; a \in A\}$ and $c>1$?
\end{itemize}

\subsection{Size of $\mathcal{A}^\mathrm{uni}$}

A typical example of a set in $\mathcal{M}$ that does not have natural density, but is assigned a probability by $\alpha$, is
\begin{equation}
A := \bigcup_{n=0}^\infty [e^{2n},e^{2n+1}),
\end{equation}
for which we have $\alpha(A)=1/2$. It is, however, unclear how ``many'' of such sets there are, i.e. how much ``bigger'' the $f$-system $\mathcal{A}^\mathrm{uni}$ is than $\mathcal{C}$ and how much ``smaller'' it is than $\mathcal{M}$. If we could construct a uniform probability measure on $\mathcal{M}$ by the method of Section \ref{sec:general}, we could determine the probability of $\mathcal{A}^\mathrm{uni}$ if $\mathcal{A}^\mathrm{uni} \in \mathcal{A}^\mathrm{uni}(\mathcal{M}) $. To construct such a probability measure, we need to equip $\mathcal{M}$ with a metric $d$ such that $(\mathcal{M},d)$ has a uniform measure. It is, however, not at all clear how we should choose $d$. So at this point, it is not clear if there is a useful way of measuring the collections $\mathcal{C}$ and $\mathcal{A}^\mathrm{uni}$.

\end{document}